\documentclass[twoside,a4paper,12pt]{article}
\usepackage{amsmath,amsthm,amsfonts,amssymb}
\usepackage{mathrsfs}
\makeatletter
\def\@seccntformat#1{\csname the#1\endcsname.\quad}
\renewcommand\section{\@startsection {section}{1}{\z@}%
                                   {-3.5ex \@plus -1ex \@minus -.2ex}%
                                   {2.3ex \@plus.2ex}%
                                   {\normalfont\bf\center }}
\renewcommand\subsection{\@startsection {subsection}{1}{\z@}%
                                   {-3.5ex \@plus -1ex \@minus -.2ex}%
                                   {2.3ex \@plus.2ex}%
                                   {\normalfont\bf}}
\makeatother

\newcommand{\bR}{\mathbf{R}}  
\newcommand{\bB}{\mathbf{B}} 

\newcommand{\cL}{\mathcal{L}}
\newcommand{\sF}{\mathscr{F}}
\newcommand{\sG}{\mathscr{G}}

\newcommand{\la}{\langle}
\newcommand{\ra}{\rangle}

\newcommand{\vol}{\operatorname{vol}}
  \newtheorem{thm}{Theorem}[section]
  
  \newtheorem{lemma}[thm]{Lemma}
  \newtheorem{prop}[thm]{Proposition}
  \newtheorem{rem}[thm]{Remark}

\theoremstyle{definition}

\numberwithin{equation}{section}

\begin{document}

\title{Brownian Hitting to Spheres}
\author{Yuji Hamana and Hiroyuki Matsumoto} 
\date{}
\maketitle

\begin{abstract} 
Let $S^{d-1}_r$ be the sphere in $\bR^d$ 
whose center is the origin and the radius is $r$, and 
$\sigma_r$ be the first hitting time to it of the standard Brownian motion 
$\{B_t\}_{t\geqq0}$, possibly with constant drift.   
The aim of this article is 
to show explicit formulae by means of spherical harmonics 
for the density of the joint distribution of $(\sigma_r,B_{\sigma_r})$ 
and to study the asymptotic behavior of the distribution function.  
\footnote{
{\it keywords}\; :\; Brownian motion, hitting times and places, 
spherical harmonics, one-dimensional diffusion, 
}
\end{abstract}

\section{Introduction and main results}

For $d\geqq2$, we consier a standard $d$-dimensional Brownian motion 
$B=\{B_t\}_{t\geqq0}$ starting from a fixed point $(a,0,...,0)$, 
where we assume $a>0$, defined on a probability space $(\Omega,\sF,P_a)$.   
Letting $S_r^{d-1}$ be the sphere in $\bR^d$ with radius $r$ 
and centered at the origin, 
we are concerned with the joint distribution of 
the first hitting time $\sigma_r$ of $B$ to $S_r^{d-1}$ and 
the hitting place $B_{\sigma_r}$.   

The aim of this article is to show an explicit expression for 
the density of the joint distribution by means of the spherical harmonics, 
that is, the Gegenbauer and the Chebyshev polynomials. 
As an application, we study the asymptotic behavior of the tail probability 
$P_a(t < \sigma_r<\infty, B_{\sigma_r}\in A),A\subset S_b^{d-1}$ when $a>r$.  
The joint density for the Brownian motion with constsnt drift 
is also investigated.   

Several authors have studied the joint distribution.   
It should be first noted that in the exit problem, that is the case of $a<r$, 
the joint density is given by a solution for a heat equation 
with the Dirichlet boundary condition.   
See Aizenman-Simon~\cite{AS} for general discussion 
and Hsu~\cite{Hsu} for an explicit expression in the case of spheres.   
Wendel~\cite{Wendel} has shown a nice result on the expectations 
of functions of $(\sigma_b, B_{\sigma_b})$ by using the spherical harmonics.   
See Gzyl~\cite{Gzyl} and references therein 
for a recent study on this direction and 
Uchiyama \cite{Uchi1,Uchi2} on the asymptotic behavior of 
the distribution functions and its application to the Wiener sausage.   
A similar problem for a Brownian motion with drift 
has been discussed in Yin-Wang~\cite{YinWang}.  

We proceed to a different way.   
Starting from the skew-product representation of Brownian motion, 
we use the fact due to Mijatovic-Mramor-Uribe Bravo~\cite{MiMrUr} 
that the projections of the Brownian motion on the sphere 
$S^{d-1}=S^{d-1}_1$ define diffusion processes.   
We see that, for one-dimensional projections, the eigenvalues and 
the eigenfunctions for the generators are explicitly given 
by the spherical harmonics.  

Combining these facts with the rotation invariance of 
the probability law of Brownian motion, we show the following.  
As usual we denote by $I_\nu$ and $K_\nu$ the modified Bessel functions.  
We also denote by $C^\nu_n$ and $T_n$ the Gegenbauer and the Chebyshev 
polynomials, respectively.   

\begin{thm} \label{t:LT}
Denote by $E_a$ the expectation with respect to $P_a$.   
Then, for $\lambda>0$ and $u\in\bR^d$, we have 
\begin{equation*} \begin{split}
E_a[e^{-\lambda\sigma_r} e^{\la u,B_{\sigma_r} \ra}] = 
 & \frac{\cL_0(a\sqrt{2\lambda})}{\cL_0(r\sqrt{2\lambda})} 
\int_{S^1} e^{r\la u, z\ra} ds(z) \\
 & + 2 \sum_{n=1}^\infty 
\frac{\cL_n(a\sqrt{2\lambda})}{\cL_n(r\sqrt{2\lambda})} 
\int_{S^1}e^{r \la u,z\ra} T_n(z_1) ds(z)
\end{split} \end{equation*}
when $d=2$ and 
\begin{equation*} \begin{split}
E_a\Big[ & e^{-\lambda\sigma_r} e^{\la u,B_{\sigma_r}\ra} 
I_{\{\sigma_r<\infty\}}\Big] \\ 
 & = 
\frac{1}{\nu} \sum_{n=0}^\infty (n+\nu) 
\frac{a^{-\nu}\cL_{n+\nu}(a\sqrt{2\lambda})}
{r^{-\nu}\cL_{n+\nu}(r\sqrt{2\lambda})}
\int_{S^{d-1}} e^{r \la u,z \ra} C_n^\nu(z_1) ds(z) 
\end{split} \end{equation*}
when $d\geqq 3$, where $ds$ is the uniform probability measure on $S^{d-1}$, 
and $\cL=I$ for $a<r$ and $\cL=K$ for $a>r$. 
\end{thm}

Setting $u=0$ and 
noting that the surface integrals of $T_n(z_1)$ and $C_n^\nu(z_1)$ 
vanish for $n\geqq1$, we recover the well known formula 
for $E_a[e^{-\lambda\sigma_r}]$ (cf. \cite{BS}).

We can invert the joint Laplace transform and obtain the following.  
We denote by $\rho_{a,r}^{(\nu)}(t)$ the probability density 
of the first hitting time to $r$ of a Bessel process with index $\nu$ 
starting from $a$.   

\begin{thm} \label{t:D}
For $t>0$ and $z\in\bR^d$ with $|z|=r$, we have 
\begin{equation*}
P_a(\sigma_r\in dt, B_{\sigma_r}\in dz) = 
\rho_{a,r}^{(0)}(t) dt ds_r(z) + 
2 \sum_{n=1}^\infty \big(\frac{a}{r}\big)^n \rho_{a,r}^{(n)}(t) 
T_n\big(\frac{z_1}{r}\big) dt ds_r(z)
\end{equation*}
when $d=2$ and 
\begin{equation*} \begin{split} 
P_a(\sigma_r\in dt, & B_{\sigma_r}\in dz) \\ 
 & = \frac{1}{\nu} \sum_{n=0}^\infty \big( n+\nu \big) 
\big(\frac{a}{r}\big)^n \rho_{a,r}^{(n+\nu)}(t)  
C_n^{\nu}\big( \frac{z_1}{r} \big) dt ds_r(z)
\end{split} \end{equation*}
when $d\geqq3$, where $\nu=\frac{d-2}{2}$ and 
$ds_r$ is the uniform probability measire on $S^{d-1}_r$.
\end{thm}

The authors~\cite{ProcAMS} have shown another expression 
for the joint Laplace transform, from which we can prove Theorem \ref{t:D}.   

The rest of this article is organized as follows.   
In the next Section 2 we study the first coordinate or the one-dimensional 
projection of the Brownian motion on $S^{d-1}$.   
We give proofs of the theorems mentioned above in Section 3 and, 
the asymptotic behavior of $P_a(t < \sigma_r<\infty,B_{\sigma_r}\in A)$, 
$A\subset S_r^{d-1}$, as $t\to\infty$ is investigated in Section 4.  
In the final Section 5, 
we deal with the Brownian motion with constant drift.   


\section{Projection of Brownian motion on sphere}

Let $\theta=\{\theta(t)\}_{t\geqq0}$ be a Brownian motion on $S^{d-1}$, 
which corresponds to the Laplace-Beltrami operator on $S^{d-1}$, 
endowed with the usual Euclidean metric.   
Mijatovic-Mramor-Uribe Bravo~\cite{MiMrUr} has shown that 
the projections of $\theta$ are diffusion processes which are realized 
as unique solutions of stochastic differential equations.   
This fact, especially on the one-dimensional projections, is fundamental 
in our argument and we recall the result in this special case.   

\begin{prop} \label{p:proj} 
The first coordinate $\{\theta_1(t)\}_{t\geqq0}$ of $\theta$ is 
a diffusion process on $(-1,1)$ whose generator is 
\begin{equation*}
\sG_d = \frac12 (1-x^2) \frac{d^2}{dx^2} - \frac{d-1}{2}x\frac{d}{dx}.
\end{equation*}
\end{prop}

We see easily that 
the boundaries $\pm 1$ are regular and reflecting when $d=2$ and 
they are entrance ones when $d\geqq 3$.   
The eigenvalues and the eigenfunctions of $\sG_d$ are explicitly given and 
we have the eigenfunction expansion for the transition densities.   
Since these play important roles in the following sections, 
we now recall some fundamental facts.   
For details of the Chebyshev and the Gegenbauer polynomials below, 
we refer to \cite{GR,MOS,Muller}.   

Write 
\begin{equation*}
\sG_d = \frac{1}{2(1-x^2)^{\frac{d-3}{2}}} \frac{d}{dx} \Big( 
\frac{1}{(1-x^2)^{-\frac{d-1}{2}}} \frac{d}{dx} \Big) 
\end{equation*}
and let $dm(x)=2(1-x^2)^{\frac{d-3}{2}}dx$ be the canonical (speed) measure.   
Note that $m$ is a finite measure on $(-1,1)$.   
Moreover, we take 
\begin{equation*}
s(x)=\int_0^x (1-y^2)^{-\frac{d-1}{2}}dy
\end{equation*}
as the scale function.

When $d=2$, $s(\pm1)$ are both finite and the boundaries are regular.   
The Chebyshev polynomial $T_n(x)=\cos( n \arccos x)$ satisfies 
\begin{equation*}
\sG_2T_n=-\frac{n^2}{2}T_n \qquad \text{and} \qquad
\frac{d}{ds}T_n(\pm1)=0.
\end{equation*}
Moreover the orthogonality relation is given by 
\begin{equation*}
\int_{-1}^1 T_m(x) T_n(x) \frac{dx}{\sqrt{1-x^2}} = 
\begin{cases} 0 & m\ne n \\ \frac{\pi}{2} & m=n\ne0 \\ \pi & m=n=0 .
\end{cases} \end{equation*}
Hence, setting 
\begin{equation*}
\phi_0^0(x)=\frac{1}{\sqrt{2\pi}}, \qquad 
\phi_n^0(x)= \frac{1}{\sqrt{\pi}}T_n(x) \quad (n\geqq1),
\end{equation*}
we see that 
$\{\phi_n^0\}_{n=0}^\infty$ gives rise to an orthonormal basis of $L^2(dm)$ 
and that the transition density $p_2(t,x,y)$ of $\{\theta_1(t)\}$ 
with respect to $dm$ is given by 
\begin{equation} \label{2e:d=2}
p_2(t,x,y) = \frac{1}{2\pi} + \frac{1}{\pi} \sum_{n=1}^\infty 
e^{-\frac{1}{2}n^2t} T_n(x) T_n(y).
\end{equation}
\indent 
For $d\geqq3$, the eigenfunctions are given 
by the Gegenbauer polynomials $C_n^\nu$ defined by 
\begin{equation*}
\sum_{n=0}^\infty s^n C_n^\nu(x) = \frac{1}{(1+s^2-2sx)^\nu}, \qquad |s|<1,
\end{equation*}
where $\nu=(d-2)/2$.   
In fact, we have 
\begin{equation*}
\sG_d C_n^\nu = -\frac12 n(n+2\nu)C_n^\nu \qquad \text{and} \qquad 
\frac{d}{ds}C_n^\nu(\pm 1)=0
\end{equation*}
and the orthogonality relation 
\begin{equation*}
\int_{-1}^1 C_m^\nu(x) C_n^\nu(x) (1-x^2)^{\nu-\frac12 }dx = \delta_{m,n} 
\frac{\pi \Gamma(n+2\nu)}{2^{2\nu-1}(n+\nu)n!(\Gamma(\nu))^2}.
\end{equation*}
Hence, setting 
\begin{equation*}
\phi_n^\nu(x) = \Big( \frac{(n+\nu)n!}{\pi\Gamma(n+2\nu)} \Big)^{\frac{1}{2}} 
2^{\nu-1}\Gamma(\nu) C_n^\nu(x),
\end{equation*}
we obtain an orthonormal basis $\{\phi_n^\nu\}_{n=0}^\infty$ of $L^2(dm)$ and 
an eigenfunction expsnsion for the transition density $p_d(t,x,y)$ 
of $\{\theta_1(t)\}$ with respect to $dm$, 
\begin{equation} \label{2e:d}
p_d(t,x,y) = \sum_{n=0}^\infty e^{-\frac{1}{2}n(n+2\nu)t} 
\phi_n^\nu(x)\phi_n^\nu(y).
\end{equation}


\section{Proof of Theorems \ref{t:LT} and \ref{t:D}}

We use the same notation as those in Section 1 and start the argument 
from the skew-product representation of the standard Brownian motion 
$B=\{B_t\}_{t\geqq0}$: there exists a $d$-dimensional Bessel process 
$R=\{R_t\}_{t\geqq0}$ (with index $\nu=(d-2)/2$) and a Brownian motion 
$\theta=\{\theta(t)\}_{t\geqq0}$ on $S^{d-1}$, independent of $R$, such that 
\begin{equation*}
B_t = R_t \theta(\Xi_t), \qquad \Xi_t=\int_0^t \frac{ds}{R_s^2}.
\end{equation*}
$B_0=(a,0,...,0)$ means $R_0=a$ and $\theta(0)=(1,0,...,0)$.   
By the independence of $R$ and $\theta$, we have 
\begin{equation*}
E_a[e^{-\lambda\sigma_r} e^{\la u,B_{\sigma_r} \ra}] = 
E_a^{(\nu)}[ e^{-\lambda\tau_r}  
E_a[e^{r\la u,\theta(t) \ra}]\Big|_{t=\Xi_{\tau_r}}],
\end{equation*}
where $E_a^{(\nu)}[\ \cdot\ ]$ denotes the expectation 
with respect to the probability law of $R$ and 
$\tau_r$ is the first hitting time of $R$ to $r$.   

First we prove the theorems when $d=2$.   
Writing $\theta(t)=(\theta_1(t),\theta_2(t))$ and $u=(u_1,u_2)$, 
we have by the rotation invariance of the law of standard Brownian motion  
\begin{equation*} \begin{split} 
 E_a[e^{r \la u,\theta(t) \ra} ] & = 
E_a[e^{r u_1 \theta_1(t)} E_a[e^{r u_2 \theta_2(t)} | \theta_1(t)]] \\
 & = \int_{-1}^1 e^{ru_1y} \frac12 
\big( e^{ru_2\sqrt{1-y^2}} + e^{-ru_2\sqrt{1-y^2}} \big) P(\theta_1(t)\in dy).
\end{split} \end{equation*}
Hence formula \eqref{2e:d=2} implies 
\begin{align*}
 & E_a[e^{r \la u,\theta(t) \ra} ] \\
 & = \frac{1}{2\pi} \int_{-1}^1 e^{ru_1y} 
\frac12 \big( e^{ru_2\sqrt{1-y^2}} + e^{-ru_2\sqrt{1-y^2}} \big) 
\frac{2dy}{\sqrt{1-y^2}} \\ 
 & \quad + \frac{1}{\pi} \sum_{n=1}^\infty e^{-\frac{1}{2}n^2t} \int_{-1}^1 
e^{ru_1y} \frac12 \big( e^{ru_2\sqrt{1-y^2}} + e^{-ru_2\sqrt{1-y^2}} \big)
T_n(y) \frac{2dy}{\sqrt{1-y^2}}
\end{align*}
since $T_n(1)=1$.   
The change of order of the intengal and the sum is easily justified 
because $|T_n(y)|\leqq1$.   
We can write the integrals on the right hand side 
as surface integrals and obtain 
\begin{equation*}
E_a[e^{r \la u,\theta(t) \ra} ] = \int_{S^1}e^{r\la u,z \ra} ds(z) + 
2\sum_{n=1}^\infty e^{-\frac{1}{2}n^2t} 
\int_{S^1} e^{r\la u,z \ra} T_n(z_1) ds(z).
\end{equation*}
Now, recalling the formula (\cite[p.407]{BS})
\begin{equation} \label{3e:Bessel2}
E_a^{(0)}[ e^{-\lambda\tau_r-\frac{1}{2}n^2\Xi_{\tau_r}} ] = 
\frac{\cL_n(a\sqrt{2\lambda})}{\cL_n(r\sqrt{2\lambda})},
\end{equation}
we obtain the assertion of Theorem \ref{t:LT} when $d=2$.  

Next note another formula (\cite[p.398]{BS})
\begin{equation*}  
E_a^{(\mu)}[ e^{-\lambda\tau_r} ] = 
\int_0^\infty e^{-\lambda t} \rho^{(\mu)}_{a,r}(t)dt = 
\frac{a^{-\mu}\cL_\mu(a\sqrt{2\lambda})}{r^{-\mu}\cL_\mu(r\sqrt{2\lambda})}.
\end{equation*}
Then we obtain Theorem \ref{t:D} when $d=2$.   
Again we can easily show the absolute convergence and 
justify the change of the sum and the integrals in $t$.   

Next we prove the theorems in the case of $d\geqq3$, 
when, for the spherical Brownian motion $\theta$, 
the conditional distribution of $(\theta_2(t),...,\theta_d(t))$ 
given $\theta_1(t)=\xi_1$ is the uniform distribution 
on the sphere $S^{d-2}_{\sqrt{1-\xi_1^2}}$ with raduis $\sqrt{1-\xi_1^2}$.   
Hence, writing $u=(u_1,u'),\theta=(\theta_1,\theta')\in\bR\times \bR^{d-1}$, 
we have 
\begin{equation*}
 E_a[e^{\la u,r\theta(t)\ra}] = 
E_a\Big[ e^{ru_1\theta_1(t)} \int_{S^{d-2}} 
e^{r\sqrt{1-\theta_1(t)^2}\la u', \xi'\ra}\ 
ds(\xi')\Big]. 
\end{equation*}
By using the facts on the Gegenbauer polynomials given in the previous section 
and writing the double integral as a surface integral, 
we obtain, from \eqref{2e:d}
\begin{align*}
 & E_a[e^{\la u,r\theta(t)\ra}] \\
 & = \sum_{n=0}^\infty e^{-\frac{1}{2}n(n+2\nu)t} \phi_n^\nu(1) 
\int_{-1}^1\phi_n^\nu(\xi_1)e^{ru_1\xi_1} 2(1-\xi_1^2)^{\frac{d-3}{2}}d\xi_1 \\
 & \qquad \qquad \qquad \qquad \qquad \qquad \times 
\int_{S^{d-2}} e^{r\sqrt{1-\xi_1^2}\la u',\xi' \ra} 
\frac{\vol(d\xi')}{\vol(S^{d-2})} \\
 & = \sum_{n=0}^\infty 
\frac{(n+\nu)2^{2\nu-1}\Gamma(\nu)^2\vol(S^{d-1})}
{\pi\Gamma(2\nu)\vol(S^{d-2})} e^{-\frac{1}{2}n(n+2\nu)t} 
\int_{S^{d-1}} C_n^\nu(w_1)e^{r\la u,w\ra} ds(w).
\end{align*}
We have used the formula 
$C_n^\nu(1)=\binom{2\nu+n-1}{n}=\Gamma(n+2\nu)/(n!\Gamma(2\nu))$, 
and also the estimate 
\begin{equation} \label{3e:est} 
\max_{|y|\leqq1}|C_n^\nu(y)| = C_n^\nu(1) \leqq Cn^{2\nu-1}
\end{equation}
for some constant $C$ (see, e.g., \cite[pp.218,\;225]{MOS}) 
to justify the change of the order of the sum and the integration.   

Moreover, recalling the foumulae 
\begin{equation*}
\vol(S^{d-1}) = \frac{2\pi^{\frac{d}{2}}}{\Gamma(\frac{d}{2})} 
\qquad \text{and} \qquad 
\Gamma(2\nu)=\frac{2^{2\nu}}{2\sqrt{\pi}} \Gamma(\nu)\Gamma(\nu+\frac12 ),
\end{equation*}
we obtain 
\begin{equation*}
E_a[e^{\la u,r\theta(t)\ra}] = \frac{1}{\nu} \sum_{n=0}^\infty (n+\nu) 
e^{-\frac{1}{2}n(n+2\nu)t} \int_{S^{d-1}} C_n^\nu(w_1) e^{r\la u,w\ra} ds(w).
\end{equation*}
Now, using \eqref{3e:Bessel2}, we obtain the assertion of Theorem \ref{t:LT}.  
Theorem \ref{t:D} is proven in the same way as in the case of $d=2$.   


\section{Asymptotic behavior of distribution function}

In this section, assuming $a>r$ and applying Theorem \ref{t:D}, 
we study the asymptotic behavior of the distribution function 
$P_a(t<\sigma_r<\infty, B_{\sigma_r}\in A)$ as $t\to\infty$ 
for a fixed Borel subset $A$ of the sphere $S_r^{d-1}$.   
We use the same notation as in the previous sections.   

In a course of study on the first hitting times of Bessel processes, 
the authors \cite{HM,ECP} have shown the following.   
Consider a Bessel process with index $\nu$ and starting from $a$ 
defined on some probability space $(\Omega',\sF', Q^{(\nu)}_a)$ 
and let $\tau_r$ be its hitting time to $r$.   
Then the asymptotic behavior 
of $Q_a^{(\nu)}(t<\tau_r<\infty)$ when $a>r$ is given by 
\begin{equation} \label{e:bessel_2}
Q_a^{(0)}(t<\tau_r<\infty)=\frac{2\log(a/r)}{\log t} (1+o(1))
\end{equation}
when $d=2$ and 
\begin{equation} \label{e:bessel_d}
Q_a^{(\nu)}(t<\tau_r<\infty)=\kappa^{(\nu)} t^{-\nu} (1+o(1)), 
\end{equation}
when $d\geqq3$, where the constant $\kappa^{(\nu)}$ is given by 
\begin{equation*}
\kappa^{(\nu)}= \frac{1}{\Gamma(\nu+1)} \Big( \frac{r^3}{2a} \Big)^\nu 
\Big\{ \Big( \frac{a}{r} \Big)^\nu - \Big( \frac{a}{r} \Big)^{-\nu} \Big\}.
\end{equation*}

Applying these results with some estimates for the remainder terms, 
we show the following.

\begin{thm} \label{t:asy}
For any Borel subset $A$ of $S^{d-1}_r$, 
\begin{equation*}
P_a(t<\sigma_r<\infty, B_{\sigma_r}\in A) = 
\frac{2\log(a/r)}{\log t} s_r(A) (1+o(1))
\end{equation*}
holds as $t\to\infty$ when $d=2$ and 
\begin{equation*}
P_a(t<\sigma_r<\infty, B_{\sigma_r}\in A) = 
\kappa^{(\nu)} s_r(A) t^{-\nu} (1+o(1))
\end{equation*}
holds when $d\geqq3$.
\end{thm}

\begin{rem} {\rm 
For the distribition function $Q_a^{(\nu)}(t<\tau_r<\infty)$ of 
the first hitting time of the Bessel process, 
Hamana et al. \cite{HKS} has shown a precise asymptotic expansion and, 
using the results, we can show asymptotic expansions 
for our joint distribution functions.   
The details will be published elsewhere.
}\end{rem}

For a proof of Theorem \ref{t:asy}, we show the following estimate 
for the tail probability of $\sigma_r$.   

\begin{lemma} \label{l:bessel}
Assume $d\geqq3$.   Then, for $t>0$, we have 
\begin{equation*}
P_a(t<\sigma_r<\infty) \leqq \frac{r^{2\nu}}{2^\nu\Gamma(\nu+1)t^\nu}.
\end{equation*}
\end{lemma}

\begin{proof}
Let $L_r$ be the last hitting time of the Brownian motion $B$ to 
the spehere $S_r^{d-1}$:
\begin{equation*}
L_r = \sup\{s>0:|B_s|=r\}.
\end{equation*}
As usual we set $L_r=0$ when $B$ does not hit $S_r^{d-1}$.  
Then we have 
\begin{equation*}
P_a(t<\sigma_r<\infty) \leqq P_a(t<L_r<\infty).
\end{equation*}
\indent 
Denote by $\mu_r$ the equilibrium measure of the ball $\bB_r$ 
with radius $r$ and centered at the origin.   
Then it is well known (\cite{PS}) that 
\begin{equation*}
P_a(t<L_r<\infty) = \int_t^\infty ds \int_{\bR^d} \frac{1}{(2\pi s)^{d/2}} 
e^{-\frac{|x-a|^2}{2s}} d\mu_r(x).
\end{equation*}
Recalling now that the capacity of $\bB_r$ is 
$\mu_r(\bR^d)=2\pi^{\frac{d}{2}}r^{d-2}/\Gamma(\frac{d}{2}-1)$, we see 
\begin{equation*}
P_a(t<L_r<\infty) \leqq 
\int_t^\infty ds \int_{\bR^d} \frac{1}{(2\pi s)^{d/2}} d\mu_r(x) = 
\frac{r^{2\nu}}{2^\nu\Gamma(\nu+1)t^\nu}. \qedhere
\end{equation*}
\end{proof}

\begin{rem} {\rm 
For transient one-dimensional diffusion processes, 
the densities of the last hitting times are written 
by means of the transition densities.  
This is the case of the Bessel processes with dimensions $d\geqq3$ and, 
moreover, we have explicit expressions for the transition densities 
We can give another proof for Lemma \ref{l:bessel} by using these facts.
}\end{rem}

We can now give a proof of Theorem \ref{t:asy}.   
Note that the infinite sum below for the expression for the joint distribution 
is absoletely convergent.   

For $d=2$, we have by Theorem \ref{t:D}
\begin{equation*}
P_a(t<\sigma_r<\infty, B_{\sigma_r}\in A) = Q_a^{(0)}(\tau_r>t) s_b(A) + I_t,
\end{equation*}
where 
\begin{equation*}
I_t = 2 \sum_{n=1}^\infty \big(\frac{a}{r}\big)^n Q_a^{(n)}(t<\tau_r<\infty) 
\int_A T_n(\frac{z_1}{r}) ds_r(z).
\end{equation*}
Assume $t>1$ and note $|T_n(x)|=|\cos(n\arccos x)|\leqq 1$.   
Then, by Lemma \ref{l:bessel}, we get 
\begin{equation*}
|I_t|\leqq  
2 \sum_{n=1}^\infty \Big(\frac{a}{r}\Big)^n 
\frac{r^{2n}}{2^n \Gamma(n+1)t^n}
\leqq \frac{2}{t} \sum_{n=1}^\infty \Big(\frac{ar}{2}\Big)^n \frac{1}{n!} 
\leqq \frac{2}{t}e^{\frac{ar}{2}}
\end{equation*}
and, by \eqref{e:bessel_2}, the assertion of Theorem \ref{t:asy}.

For $d\geqq3$, we have 
\begin{equation*}
P_a(t<\sigma_r<\infty, B_{\sigma_r}\in A) = 
Q_a^{(\nu)}(t<\tau_r<\infty) s_r(A) + J_t,
\end{equation*}
where 
\begin{equation*}
J_t = \frac{1}{\nu} \sum_{n=1}^\infty (n+\nu) \big(\frac{a}{r}\big)^n 
Q_a^{(n+\nu)}(t<\tau_r<\infty) \int_A C_n^\nu(\frac{z_1}{r}) ds_r(z).
\end{equation*}
Hence, combining \eqref{3e:est} 
with Lemma \ref{l:bessel} and \eqref{e:bessel_d}, 
we see $J_t=O(t^{-1-\nu})$ and the assertion of Theorem \ref{t:asy}.  


\section{Brownian motion with drift}

Let $B=\{B_t\}_{t\geqq0}$ be a standard $d$-dimensional Brownian motion 
starting from $x=(a,0,...,0)$ as before and, 
for a constant vector $v\in \bR^d$, $B^{(v)}=\{B^{(v)}(t)\}_{t\geqq0}$ be 
a Brownian motion with drift $v$ defined by $B^{(v)}(t)=B_t+tv$.   
We denote by $\sigma^{(v)}_r$ the first hitting time of $B^{(v)}$ 
to the sphere $S^{d-1}_r$.  

The Cameron-Martin theorem and the strong Markov property 
of Brownian motion imply 
\begin{equation*}
E_a\Big[e^{-\lambda\sigma^{(v)}_r} e^{\la u,B^{(v)}(\sigma^{(v)}_r) \ra} 
I_{\{\sigma^{(v)}_r<\infty\}} \Big] = e^{-av_1} 
E_a\Big[e^{-(\lambda+\frac{|v|^2}{2})\sigma_r} e^{\la u+v,B_{\sigma_r} \ra}
I_{\{\sigma_r<\infty\}} \Big] .
\end{equation*}
Hence we can apply Theorem \ref{t:LT} to the right hand side and 
obtain the following:  

\begin{thm} 
For $\lambda>0$ and $u\in\bR^d$, we have 
\begin{equation*} \begin{split}
E_a\Big[e^{-\lambda\sigma^{(v)}_r} e^{\la u,B^{(v)}(\sigma^{(v)}_r) \ra} 
 & I_{\{\sigma_r<\infty\}} \Big] = 
 e^{-av_1} \Bigg\{ 
\frac{\cL_0(a\sqrt{2\lambda+|v|^2})}{\cL_0(r\sqrt{2\lambda+|v|^2})} 
\int_{S^1} e^{r\la u+v, z\ra} ds(z) \\
 & + 2 \sum_{n=1}^\infty 
\frac{\cL_n(a\sqrt{2\lambda+|v|^2})}{\cL_n(r\sqrt{2\lambda+|v|^2})} 
\int_{S^1} e^{r \la u+v,z\ra} T_n(z_1) ds(z) \Bigg\}
\end{split} \end{equation*}
when $d=2$ and, when $d\geqq3$, 
\begin{equation*} \begin{split}
E_a\Big[ & e^{-\lambda\sigma^{(v)}_r} e^{\la u,B^{(v)}(\sigma^{(v)}_r)\ra} 
I_{\{\sigma_r<\infty\}}\Big] \\ 
 & = \frac{1}{\nu} e^{-av_1} \sum_{n=0}^\infty (n+\nu) 
\frac{a^{-\nu}\cL_{n+\nu}(a\sqrt{2\lambda+|v|^2})}
{r^{-\nu}\cL_{n+\nu}(r\sqrt{2\lambda+|v|^2})}
\int_{S^{d-1}} e^{r \la u+v,z \ra} C_n^\nu(z_1) ds(z). 
\end{split} \end{equation*}
\end{thm} 

We can invert the Laplace transform as before and show the following:

\begin{thm} \label{t:drift_density} 
For $t>0$ and $z\in\bR^d$ with $|z|=r$, we have 
\begin{equation*} \begin{split}
P_a(\sigma_r^{(v)}\in dt, B^{(v)}(\sigma_r^{(v)})\in dz) & = 
 e^{-av_1+\la v,z \ra} e^{-\frac{|v|^2}{2}t} 
\rho_{a,r}^{(0)}(t) dt ds_r(z) \\ 
 & + 2 e^{-av_1+\la v,z \ra} e^{-\frac{|v|^2}{2}t} 
\sum_{n=1}^\infty \big(\frac{a}{r}\big)^n \rho_{a,r}^{(n)}(t) 
T_n\big(\frac{z_1}{r}\big) dt ds_r(z)
\end{split} \end{equation*}
when $d=2$ and, when $d\geqq3$ 
\begin{equation*} \begin{split} 
P_a(\sigma^{(v)}_r & \in dt, B^{(v)}_{\sigma_r}\in dz) \\ 
 & = \frac{1}{\nu} e^{-av_1+\la v,z \ra} e^{-\frac{|v|^2}{2}t} 
\sum_{n=0}^\infty \big( n+\nu \big) 
\big(\frac{a}{r}\big)^n  \rho_{a,r}^{(n+\nu)}(t) 
C_n^{\nu}\big(\frac{z_1}{r}\big) dt ds_r(z).
\end{split} \end{equation*}
\end{thm}

Next, assuming $a>r$, we consider the asymptotic behavior 
of the distribution function 
$P(t<\sigma_r^{(v)}<\infty, B^{(v)}(\sigma_r^{(v)})\in A)$ as $t\to\infty$ 
for a fixed $A\subset S_r^{d-1}$.   
As is easily guessed as earlier, the leading term is given by the first terms 
of the right hand sides in Theorem \ref{t:drift_density}.   

\begin{thm} \label{t:drift_asy}
For any Borel subset $A$ of $S_r^{d-1}$, we have 
\begin{equation*} \begin{split}
 P_a(t< & \sigma_r^{(v)}<\infty, B^{(v)}(\sigma_r^{(v)})\in A) \\
 & = 2\log\big(\frac{a}{r}\big) e^{-av_1} \int_A e^{\la v,z\ra} ds_r(z) 
\frac{1}{t(\log t)^2} e^{-\frac{|v|^2}{2}t} (1+o(1))
\end{split} \end{equation*}
when $d=2$ and 
\begin{equation*} \begin{split}
 P_a(t< & \sigma_r^{(v)}<\infty, B^{(v)}(\sigma_r^{(v)})\in A) \\
 & = \frac{2L(\nu)}{|v|^2} e^{-av_1} \int_A e^{\la v,z\ra} ds_r(z) 
t^{-\nu-1} e^{-\frac{|v|^2}{2}t} (1+o(1))
\end{split} \end{equation*}
when $d\geqq3$, where 
\begin{equation*}
L(\nu) = \frac{r^{2\nu}}{2^\nu \Gamma(\nu)} 
\Big\{1-\Big(\frac{r}{a}\Big)^{2\nu}\Big\}.
\end{equation*}
\end{thm}

In order to estimate the higher order terms, 
we recall from \cite{ProcAMS} the asymptotic result for 
\begin{equation*}
H^{(\nu)}(t) := \int_t^\infty e^{-\frac{|v|^2}{2}s} \rho_{a,r}^{(\nu)}(s) ds,
\end{equation*}
where $\rho_{a,r}^{(\nu)}$ is the density of the first hitting time $\tau_r$ 
to $r$ of the Bessel process with index $\nu$ starting from $a$:
when $d=2$, 
\begin{equation*}
H^{(\nu)}(t) = \frac{2\log(a/r)}{t(\log t)^2} e^{-\frac{|v|^2}{2}t} (1+o(1))
\end{equation*}
and, when $d\geqq3$ 
\begin{equation} \label{5e:H_d_3}
H^{(\nu)}(t) = \frac{2L(\nu)}{|v|^2t^{\nu+1}} e^{-\frac{|v|^2}{2}t} (1+o(1)).
\end{equation}
\indent
The assertion of Theorem \ref{t:drift_asy} follows from the following lemma:   

\begin{lemma} \label{l:5_ijou}
There exists a constant $C$, depending on $|v|$ and $r$, such that 
\begin{equation*}
H^{(\nu)}(t) \leqq \frac{Cr^{2\nu}}{\Gamma(\nu)} \frac{1}{(2t)^{\nu+1}} 
e^{-\frac{|v|^2}{2}t}
\end{equation*}
holds for all $d\geqq3$.   
\end{lemma}

\begin{proof}
We use \eqref{5e:H_d_3} when $d=3$ and $d=4$, 
and assume $d\geqq5$ in the following.  

Denote by $P_y$ the $d$-dimensional Wiener measure with starting point $y$ 
and use the same notation $\sigma_r$ for the first hitting time to $S_r^{d-1}$ 
of the corresponding Brownian motion.    
Moreover, let $p(t,x,y)=(2\pi t)^{-d/2}\exp(-|y-x|^2/2t)$ be 
the Gaussian kernel and set $\alpha=|v|^2/2$ for simplicity.   
Then we have 
\begin{equation*}
H^{(\nu)}(t)= \alpha \int_t^\infty e^{-\alpha s} P_a(t<\sigma_r\leqq s)ds
\end{equation*}
and, setting $e=(1,0,...,0)$, 
\begin{align*}
P_a(t<\sigma_r\leqq s) & \leqq \int_{\bR^d}p(t,ae,y)P_y(\sigma_r\leqq s-t)dy \\
 & \leqq \frac{1}{(2\pi t)^{d/2}} \int_{\bR^d} P_y(\sigma_r\leqq s-t) dy
\end{align*}
by the Markov property of Brownian motion.   
Hence we get, after a simple change of variables, 
\begin{equation*}
H^{(\nu)}(t) \leqq \frac{\alpha e^{-\alpha t}}{(2\pi t)^{d/2}} 
\int_0^\infty e^{-\alpha s}ds \int_{\bR^d} P_y(\sigma_r\leqq s)dy.
\end{equation*}
\indent
Now let $L_r$ be the last hitting time of the Brownian motion to $S_r^{d-1}$.  
Then we have 
\begin{equation}  \label{e:devide}
\int_{\bR^d} P_y(\sigma_r\leqq s) = \int_{\bR^d} P_y(0 < L_r \leqq s)dy + 
\int_{\bR^d} P_y(\sigma_r \leqq s < L_r)dy.
\end{equation}
For the second term of the right hand side, Le Gall \cite{le_gall} has shown 
\begin{equation} \label{e:le_gall} 
\int_{\bR^d} P_y(\sigma_r \leqq s < L_r)dy = 
\int_{\bR^d} P_y(\sigma_r \leqq s) P_y(\sigma_r < \infty)dy, 
\end{equation}
which implies 
\begin{equation} \label{e:suffice}
\int_{\bR^d} P_y(\sigma_r \leqq s < L_r)dy \leqq  
\int_{\bR^d} P_y(\sigma_r < \infty)^2 dy.
\end{equation}
This estimate is sufficient for our purpose.   
We give another elementary proof of \eqref{e:le_gall} 
after completing the proof of Lemma \ref{l:5_ijou}.

As in the previous section, we denote by $\mu_r$ the equilibrium measure 
of the ball $\bB_r$.   
Then we have, for the first term of \eqref{e:devide}, 
\begin{align*}
 \int_0^\infty e^{-\alpha s}ds \int_{\bR^d} P_y(0 < L_r \leqq s) dy 
 & = \int_0^\infty e^{-\alpha s}ds \int_{\bR^d}dy \int_0^s d\tau
\int_{\bR^d} p(\tau,y,z) d\mu_r(z) \\
 & = \int_0^\infty e^{-\alpha s}ds \int_0^s d\tau \int_{\bR^d} d\mu_r(z) \\
 & = \frac{2\pi^{d/2}r^{d-2}}{\alpha^2\Gamma(\frac{d}{2}-1)}.
\end{align*}
For the second term, we recall 
\begin{equation*}
P_y(\sigma_r<\infty) = 1\wedge\Big(\frac{r}{|y|}\Big)^{d-2}.  
\end{equation*}
Then, by \eqref{e:suffice}, we get 
\begin{align*}
 \int_0^\infty e^{-\alpha s}ds \int_{\bR^d} P_y(\sigma_r<\infty)^2 dy 
 & = \frac{1}{\alpha} \Big\{ \int_{|y|\leqq r} dy + 
\int_{|y|\geqq r} \big(\frac{r}{|y|}\big)^{2(d-2)}dy \Big\} \\
 & = \frac{2\pi^{\frac{d}{2}}r^d}{\alpha\Gamma(\frac{d}{2})} 
\Big( \frac{1}{d} + \frac{1}{d-4} \Big).
\end{align*}
\indent
Combining the above inequalities, 
we obtain the assertion of the lemma.
\end{proof}

\begin{proof}[{\it Proof of} \eqref{e:le_gall}]
By the Markov property of Brownian motion, we have 
\begin{align*}
 \int_{\bR^d} P_y(\sigma_r \leqq s < L_r) dy 
 & = \int_{\bR^d} E_y[1_{\{\sigma_r\leqq s\}} 1_{\{L_r>s\}}]dy \\
 & = \int_{\bR^d} E_y[1_{\{\sigma_r\leqq s\}} 
E_{B_s}[1_{\{L_r>0\}}]]dy \\
 & = \int_{\bR^d}dy \int_{\bR^d} 
E_y[ 1_{\{\sigma_r\leqq s\}} P_{B_s}(L_r>0) |B_s=x] p(s,y,x)dx \\
 & = \int_{\bR^d} dx \int_{\bR^d} P_x(L_r>0) 
P_y(\sigma_r\leqq s|B_s=x) p(s,y,x) dy.  
\end{align*}
Note here that $P_x(L_r>0)=P_x(\sigma_r<\infty)$ and 
that the time reversal of a pinned Brownian motion is again a pinned 
Brownian motion.   
Then we obtain 
\begin{align*}
\int_{\bR^d} P_y(\sigma_r \leqq s < L_r) dy & = 
 \int_{\bR^d} P_x(\sigma_r<\infty) dx 
\int_{\bR^d} P_x(\sigma_r\leqq s|B_s=y) p(s,x,y) dy \\
 & = \int_{\bR^d} P_x(\sigma_r<\infty) P_x(\sigma_r \leqq s) dx. \qedhere
\end{align*}
\end{proof}


\noindent{\bf Acknowledgment}

The authors were partially supported by JSPS KAKENHI 
Grant Numbers 20K03634 and 21K03298.

\noindent Yuji Hamana \\
hamana@math.tsukuba.ac.jp \\
Department of Mathematics \\
University of Tsukuba \\
1-1-1 Tennodai, Tsukuba 305-8571, Japan \\

\bigskip

\noindent Hiroyuki Matsumoto \\
matsu@math.aoyama.ac.jp \\
Department of Mathematics \\
Aoyama Gakuin University \\
Fuchinobe 5-10-1, Sagamihara 252-5258, Japan

\end{document}